\documentclass[12pt]{amsart}
\newtheorem{thm}{Theorem}[section]
\newtheorem{lem}[thm]{Lemma}
\newtheorem{cor}[thm]{Corollary}
\newtheorem{prop}[thm]{Proposition}

\theoremstyle{definition}

\theoremstyle{remark}

\begin{document}

\title[Orthogonal Sudoku Squares]{A Quick Construction of Mutually Orthogonal Sudoku Squares}
\author{John Lorch}
\address{Department of Mathematical Sciences\\ Ball State University\\Muncie, IN  47306-0490}
\email{jlorch@bsu.edu} 
\date{December 2010}
\begin{abstract}
For any odd prime power $q$ we provide a quick construction of a complete family of $q(q-1)$ mutually orthogonal sudoku squares of order $q^2$.
\end{abstract}
\maketitle

\def\rk{{\rm rank}}
\def\A{{\mathcal A}}
\def\reg{{\mathcal R}}
\def\P{{\mathcal P}}
\def\F{{\mathbb F}}
\def\fF{{\mathcal F}}
\def\L{{\mathcal L}}
\def\R{{\mathbb R}}
\def\C{{\mathbb C}}
\def\Z{{\mathbb Z}}
\def\B{{\mathcal B}}
\def\Xorth{{\mathcal X}_{{\rm orth}}}
\def\Xnp{{\mathcal X}_{{\rm np}}}
\def\S{{\mathcal S}}
\def\Q{{\mathcal Q}}
\def\pt#1{{\langle  #1 \rangle}}
\def\ds{\displaystyle}
\def\MF{M^{2\times 2}(\F )}
\def\GL{{\rm GL}(2,q)}
\def\Gr{{\rm Gr}(2,\F ^4)}

\section{Introduction}

We present a quick method for producing a complete orthogonal family of sudoku squares of order $q^2$, where $q$ is an odd prime power.
The method, which involves identifying parallel linear
sudoku squares with $2\times 2$ matrices, is simple and
constructive.  Ideas supporting this method originate in
\cite{rB08}, \cite{jL09}, and \cite{jL10}, and may be viewed as a complement to those in \cite{jL09} where the author chose to work in a more specialized setting.

A {\bf sudoku square} is a latin square of order $n^2$ with an additional
requirement called the {\bf subsquare condition}: Upon partitioning
the array into $n\times n$ {\bf subsquares}, each subsquare must contain
every symbol. (See Figure \ref{f:locationex} for an example.) Also, two latin squares are said
to be {\bf orthogonal} if upon superimposition of the two squares there is no repetition of ordered pairs of symbols.
Interest in families of orthogonal latin squares dates from Euler \cite{lE82} and continues to be fueled by
statistical applications and intriguing open problems involving the maximum size of a family of mutually orthogonal latin squares (see \cite{cC01} or \cite{rB08}).

The maximum size of a family of mutually orthogonal sudoku
squares of order $q^2$ is known to be $q(q-1)$. This bound has been achieved
in several ways: In \cite{tV09} a complete family of this size is produced by permuting the rows of an addition
table for the finite field of order $q^2$, thus channeling the classical methods of Bose, Moore, and Stevens (see \cite{rB38} and \cite{aK10}).
Another approach taken in \cite{rB08} rests on three-dimensional
projective geometry: There the authors consider `parallel linear' sudoku
squares that are each characterized by a projective line; the
complete orthogonal family is identified with all lines in a
regular spread that fail to meet a certain regulus. The approach taken in this article, begun in
\cite{jL10} for $q$ prime and considerably refined and strengthened here due to inspiration from
\cite{rB08}, relies on representing
parallel linear sudoku squares by $2\times 2$-matrices with entries in the finite field $\F$ of order $q$.

The structure of the article is as follows. Background on the representation and
orthogonality of parallel linear sudoku squares is given in Section 2; the quick construction of a complete
family of orthogonal sudoku squares follows in Section 3. Terminology used to describe various aspects
of sudoku squares (location, large row, subsquare, etc.) follows that of \cite{rB08}, in particular:
\begin{itemize}
\item  a {\bf large row} consists of
a row of $q$ subsquares, and similarly for {\bf large column}.
\item a {\bf mini row} consists of a row of $q$ locations within a given subsquare, and similarly for {\bf mini column}.
\end{itemize}

\section{Background on Parallel Linear Sudoku Squares: Representation and Orthogonality}

\subsection{Identifying sudoku squares with $2\times 2$ matrices}
Let $\F$ be the finite field of order
$q$ with ${\rm char} ( \F )\ne 2$. Locations within a sudoku square of order $q^2$ can be
identified with the vector space $\F ^4$ over $\F$. Each
location has an address $(x_1,x_2,x_3,x_4)\in \F ^4$ (denoted $x_1x_2x_3x_4$ hereafter), where $x_1$ and $x_3$
denote the large row and column of the location, respectively,
while $x_2$ and $x_4$ denote the mini row and mini column of the
location, respectively. Large rows are labeled in
increasing lexicographic order from top to bottom starting from zero, and within a given subsquare the mini rows are similarly labeled in
increasing lexicographic order from top to bottom. Columns, both large and
mini, are labeled similarly from left to right.\footnote{More on this lexicographic order: Suppose that $q=p^k$ for some prime $p$ and that $\Z _p=\{0, 1,\dots ,p-1\}$ is the corresponding field of order $p$.
We can turn $\Z _p$ into an ordered set by declaring $0<1<\cdots < p-1$. This order on $\Z _p$ is used to impose a lexicographic order on $\F$, which can be viewed as the set of
$k$-tuples with entries $\Z _p$. In the case that $q=p$ (as in Figure \ref{f:locationex} where $q=p=3$) we simply use $0,1,\dots ,p-1$ as labels for large rows, rows within large rows, etc., as one might expect.} (See the asterisked symbol in
Figure \ref{f:locationex}.)

\begin{figure}[h]
      $$ {\small \begin{array}{| c  c  c | c c c | c c c|}
   \hline
      0 & 1 & 2 &  4 & 5 &3 & 8 & 6 & 7 \\
      3 & 4 & 5 &  7 & 8 &6 & 2 & 0 & 1^* \\
      6 & 7 & 8 &  1 & 2 &0 & 5 & 3 & 4 \\ \hline
      1 & 2 & 0 &  5 & 3 &4 & 6 & 7 & 8 \\
      4 & 5 & 3 &  8 & 6 &7 & 0 & 1 & 2 \\
      7 & 8 & 6 &  2 & 0 &1 & 3 & 4 & 5 \\ \hline
      2 & 0 & 1 &  3 & 4 &5 & 7 & 8 & 6 \\
      5 & 3 & 4 &  6 & 7 &8 & 1 & 2 & 0 \\
      8 & 6 & 7 &  0 & 1 &2 & 4 & 5 & 3 \\ \hline
    \end{array}}
      $$ 
       \caption{A parallel linear sudoku square generated by $\pt{1002,0212}\subset \Z _3 ^4$ with
       asterisked symbol in location $0122$.}  
       \label{f:locationex}
       \end{figure}

We say that a sudoku square is {\bf linear} if the collection of locations housing any given symbol
is a {\bf coset} (i.e., an additive translate) of some two-dimensional vector subspace of $\F^4$. Linear sudoku squares come in two flavors:
 If every such coset originates from a {\em single} two-dimensional
subspace, then the square is of {\bf parallel type}, otherwise the
square is of {\bf non-parallel type}.\footnote{We follow \cite{rB08} in using the terminology {\em parallel type} and {\em nonparallel type}.}  In this article we focus only
on linear sudoku squares of parallel type. For example, the sudoku square implied by Figure
\ref{f:locationex} is linear of parallel type, generated by the two-dimensional subspace
$g=\pt{1002,0212}\subset \Z _3^4$.

In order to generate a parallel
linear sudoku square we require that cosets of $g$ intersect
each row, column, and subsquare exactly once:

\begin{prop} \label{p:sudokucharacterization}
A two-dimensional subspace $g$ of $\F ^4$ generates a linear sudoku square $M_g$ (unique up to
relabeling) of parallel
type if and only if $g$ has trivial intersection with
\begin{align*}
g_c&=\pt{1000,0100},\\
g_r&=\pt{0010,0001},  \text{ and }\\
g_{ss}&=\pt{0100,0001}.
\end{align*}
\end{prop}

\begin{proof}
The proof relies on the fact that a pair of two-dimensional subspaces
of $\F ^4$ intersect trivially if and only if any two cosets of
these planes intersect in a single vector in $\F ^4$.
For the forward implication,
observe that the cosets of $g_c$
correspond to the columns of an array. Since $g$ and $g_c$ intersect trivially,
 we know that any coset of $g$
meets any coset of $g_c$ in a single vector, corresponding to a
location in an array. Therefore every coset of $g$ meets the
columns of the array in a single location, and, since each coset
of $g$ houses a constant symbol for $M_g$, we conclude that each
symbol of $M_g$ is contained in each column in a single location.
Similarly, since $g\cap g_r$ and $g\cap g_{ss}$ are both trivial, each symbol of $M_g$ is
contained in each row  and each subsquare in a single location. Therefore $M_g$ is a
sudoku square.

For the reverse implication, if $M_g$ is a sudoku square then each coset of
$g$ must intersect each coset of $g_c$, $g_r$, and $g_{ss}$
in exactly one vector. Therefore $g$ has trivial intersection with both $g_c$, $g_r$, and
$g_{ss}$.
\end{proof}

All linear sudoku squares of parallel type can be represented by
$2\times 2$ matrices. If  $A,B\in \MF$ we let $\left[
\begin{array}{c} A
\\ B
\end{array}\right]$ denote the subspace of $\F ^4$ spanned by the
columns of the matrix $\begin{pmatrix} A \\ B\end{pmatrix}$. Also
let $I$ denote the $2\times 2$ identity matrix. In consideration
of Proposition \ref{p:sudokucharacterization}, we have the
following:

\begin{prop}\label{p:sudmatrep}
A two-dimensional subspace $g$ of $\F ^4$ generates a linear sudoku square of
parallel type if and only if there exists a non-lower triangular
invertible matrix $C$ such that $g=\left[ \begin{array}{c} I\\ C
\end{array} \right]$.
\end{prop}

\begin{proof}
Let $A,B\in \MF$ such that $g=\left[
\begin{array}{c} A \\ B  \end{array}\right]$, and assume that $g$
generates a linear sudoku square of parallel type. Then $A$ and
$B$ must be invertible  to guarantee that $g$ has trivial intersection with both $g_r$
and $g_c$, respectively (see Proposition
\ref{p:sudokucharacterization}). Therefore
$$
g=\left[
\begin{array}{c} A \\ B  \end{array}\right]=\left[\begin{pmatrix} A \\ B  \end{pmatrix}A^{-1}\right]
=\left[
\begin{array}{c} I \\ BA^{-1}  \end{array}\right],
   $$
and we choose $C=BA^{-1}$. The matrix $C$ is invertible, and must
also be non-lower triangular or else the second column of
$\begin{pmatrix} I \\ C\end{pmatrix}$ will lie in $g\cap g_{ss}$,
contradicting the fact that $g$ and $g_{ss}$ must have trivial
intersection
(Proposition \ref{p:sudokucharacterization}).

On the other hand, given an invertible, non-lower triangular
matrix $C$, the plane $g=\left[ \begin{array}{c} I\\ C
\end{array} \right]$ satisfies the conditions of Proposition
\ref{p:sudokucharacterization}, so $g$ generates a linear sudoku
square of parallel type.
\end{proof}

\subsection{Orthogonality}
Two sudoku squares are said to be {\bf orthogonal} if, upon
superimposition, each ordered pair of symbols occurs exactly once.
There is a simple geometric condition that characterizes
orthogonality of parallel linear sudoku squares:

\begin{prop} \label{p:orthogcond}
Let $M_{g},M_{h}$ be linear sudoku squares
of parallel type generated by two-dimensional subspaces $g,h$ of $\F ^4$, respectively.
The two squares are orthogonal if and only if $g$ and $h$ have trivial intersection.
\end{prop}

\begin{proof}
The proof is similar to that for Proposition \ref{p:sudokucharacterization}. Let $a,b$  be any two symbols used
in the squares. Since the squares are linear of parallel type there
are cosets $x+g$ and $y+h$  of $g$ and $h$ whose
elements determine the locations of the symbol $a$ in $M_{g}$ and
of $b$ in $M_{h}$, respectively. If $g$ and $h$ have
trivial intersection
we know that $(x+g)\cap (y+h)$
consists of a single vector. Therefore there is exactly one
location that contains both $a$ in $M_{g}$ and $b$ in $M_{h}$,
so when the two squares are superimposed there is precisely one
location housing the ordered pair $(a,b)$. Therefore the squares are
orthogonal. Likewise, if the squares are orthogonal then no two
cosets of $g$ and $h$ can meet in anything other than a single
vector, else an ordered pair of symbols $(a,b)$ will either
appear more than once or not at all. Therefore $g$ and $h$ have trivial intersection.
\end{proof}

Proposition \ref{p:orthogcond} together with Proposition
\ref{p:sudmatrep} implies

\begin{cor}\label{c:orthogcond}
The two-dimensional subspaces $g_1=\left[ \begin{array}{c} I\\ C_1
\end{array} \right]$ and $g_2=\left[ \begin{array}{c} I\\ C_2
\end{array} \right]$ generate orthogonal linear sudoku squares
of parallel type if and only if $C_1,C_2$ satisfy the conditions of
Proposition \ref{p:sudmatrep} and  $\det (C_1-C_2)\ne 0$.
\end{cor}

\begin{proof}
By Proposition \ref{p:orthogcond} the planes $g_1$ and $g_2$ generate orthogonal sudoku squares if and
only if $g\cap h$ is trivial. Observe
$$
g\cap h \text{ trivial }\Longleftrightarrow \det \begin{pmatrix} I & I \\ C_1 & C_2 \end{pmatrix}\ne 0
\Longleftrightarrow \det (C_1-C_2)\ne 0.
   $$
\end{proof}

\section{Quick Construction of a Complete Orthogonal Family}
We now produce a complete mutually orthogonal family of sudoku
squares of order $q^2$ (i.e., we produce a family of size
$q(q-1)$). We begin with a classical result about quadratic
residues:

\begin{lem} \label{l:residues}
If ${\rm char} (\F )\ne 2$ then there exists $\alpha\in \F$ such
that $\alpha$ is a square in $\F$ but $\alpha +1$ is not.
\end{lem}

Apparently such values of $\alpha $ are quite common, occurring approximately
$(q-1)/4$ times in a field of order $q$ (see \cite{nR75}).

Select $\alpha$ as in Lemma \ref{l:residues} and pick $\lambda
\in \F$ such that $\lambda ^2= 2^2\cdot \alpha$. Define
$$
\fF= \left\{ \begin{pmatrix} v &w \\ w &\lambda w +v \end{pmatrix}\in
\MF \mid v\in \F \text{ and } w\in \F ^*\right\}.
   $$

\begin{thm}\label{t:construction}
The matrices in $\fF$ generate a family of $q(q-1)$ mutually
orthogonal sudoku squares of order $q^2$.
\end{thm}

\begin{proof} We show that the matrices in $\fF$ satisfy the
conditions of Proposition \ref{p:sudmatrep} and Corollary
\ref{c:orthogcond}. It is clear that $\fF$ is of size $q(q-1)$ and that
the elements of $\fF$ are non-lower triangular. It remains to show
that elements of $\fF$ are nonsingular, as are differences of
distinct elements of $\fF$.

Suppose $\begin{pmatrix} v &w \\ w &\lambda w +v
\end{pmatrix}\in \fF$. Observe that

\begin{align*}
\det \begin{pmatrix} v &w \\ w &\lambda w +v
\end{pmatrix} =0& \Longleftrightarrow w^2(1+\lambda ^2 (2^{-1})^2)
\text{ is a square in } \F \\
&\Longleftrightarrow 1+\lambda ^2 (2^{-1})^2 \text{ is a square in } \F
\\
& \Longleftrightarrow 1+\alpha \text{ is a square in } \F ,
\end{align*}
and that the lattermost statement, obtained by completing the square in
either $v$ or $w$, contradicts our choice of $\alpha $. We conclude that elements of
$\fF$ are nonsingular.

Given distinct $\begin{pmatrix} v_1 &w_1 \\ w_1 &\lambda w_1 +v_1
\end{pmatrix}, \begin{pmatrix} v_2 &w_2 \\ w_2 &\lambda w_2 +v_2
\end{pmatrix} \in \fF$, their difference
$$
\begin{pmatrix} v_1-v_2 &w_1-w_2 \\ w_1-w_2 &\lambda( w_1-w_2)
+(v_1-v_2) \end{pmatrix}
   $$ is again an element of $\fF$ if $w_1\ne
w_2$ and is hence nonsingular. If $w_1=w_2$ then $v_1\ne v_2$ and
the difference matrix is diagonal with nonzero diagonal
entries. Therefore the difference matrix is nonsingular.
\end{proof}

\end{document}